\documentclass{article}
\pdfoutput=1
\usepackage[pdfencoding=auto, psdextra, pagebackref, breaklinks]{hyperref}

\usepackage{L_func}
\usepackage[title]{appendix}

\title{The Dixmier-Malliavin Theorem and Bornological Vector Spaces}
\author{Gal Dor\\
Tel-Aviv University}
\date{January 2020}

\begin{document}

\maketitle

\begin{abstract}
    This note is intended to reformulate the Dixmier-Malliavin theorem about smooth group representations in the language of bornological vector spaces, instead of topological vector spaces. This language turns out to allow a more general theorem to be proven, and we are able to use it to strengthen a result of Meyer from \cite{borno_quasi_unital_algs2}.
    
    This paper is based on a part of the author's thesis \cite{alg_struct_L_funcs}.
\end{abstract}

\tableofcontents

\section{Introduction}

In this text, we will discuss bornological vector spaces and prove a variant of a theorem by Dixmier-Malliavin (see \cite{schwartz_is_quasi_unital} for a clean overview of the original). Along the way, we will discuss the notion of quasi-unital algebra, which frequently comes up in the representation theory of locally compact groups.

The first of these topics, bornological vector spaces, will be used in much of this text as a substitute for topological vector spaces. The two notions are very closely related (to the point that in many applications, they are indistinguishable), but sometimes that leads to theorems that should be stated in one language being stated in the other. As a result, some theorems gain spurious technical requirements that disappear once they are stated correctly. It seems that many theorems in representation theory are much cleaner when stated in the language of bornological vector spaces instead of topological vector spaces. As a result, this text was written with a hidden agenda in mind: help push more people working with representation theory into using this language.

\begin{remark}
    It is fairly common to need to place some subcategory of topological vector spaces into a ``nice'' categorical framework (usually, this means a locally presentable category). Every application has a different suitable choice of framework. For example, Scholze works with \emph{condensed sets} in \cite{condensed_sets}, a construction inspired by the pro-\'etale site.
    
    It seems to the author that bornological vector spaces are useful in the context of representation theory beyond simply being a replacement for the category of topological vector spaces with nicer formal properties. This case is heavily advocated by Meyer, e.g. in \cite{borno_quasi_unital_algs2,borno_quasi_unital_algs,borno_schwartz_padic,borno_vs_topo_analysis}. However, most of his motivation comes from the world of non-commutative geometry and homological algebra. The current text is focused on more classical motivations, by showing that the Dixmier-Malliavin theorem has a distinctly bornological ``flavor''.
\end{remark}

This brings us to the second topic we discuss in this paper: the Dixmier-Malliavin theorem. The content of our version of this theorem is that the G\r{a}rding space of a complete bornological representation is a smooth $C_c^\infty(G)$-module. Classically, as it is formulated in \cite{schwartz_is_quasi_unital}, the theorem is restricted to Fr\'echet topological representations, which are a full subcategory of the bornological representations.

Much of this text does not constitute original contribution. Certainly, everything we have to say about bornological vector spaces or spaces of smooth functions is already known, and is included here only as a reference. Many of our claims about quasi-unital algebras already appear in \cite{self_induced_algs}. Our contribution to Dixmier-Malliavin is its reformulation in terms of bornological vector spaces, which results in a cleaner and more categorical ``flavor''.

Let us give some more details. Let $G$ be a real Lie group. Recall that $C_c^\infty(G)$ is the (bornological) space of smooth and compactly supported functions on $G$, and that convolution turns it into a ring. In \cite{borno_quasi_unital_algs2}, Meyer proves an equivalence between two categories:
\[\xymatrix{
    \{\text{smooth $G$-modules}\} \ar@{<->}[r] & \{\text{essential $C_c^\infty(G)$-modules}\}.
}\]

On the left hand side of the equivalence, is the category of complete bornological $G$-modules, where the action of $G$ is \emph{smooth} (i.e., its G\r{a}rding space is equal to itself, see also Definition~\ref{def:smooth_G_module}). On the right hand side of Meyer's equivalence is the full subcategory of complete bornological $C_c^\infty(G)$-modules $M$ satisfying that the canonical morphism:
\begin{equation*} \label{eq:meyer_smoothness}
    C_c^\infty(G)\widehat{\otimes}_{C_c^\infty(G)}M\ra M
\end{equation*}
is an isomorphism. Here, $\widehat{\otimes}$ is the completion of the relative (projective) tensor product. Meyer refers to this category as the category of \emph{essential $C_c^\infty(G)$-modules}.

Note that Meyer's notion of essential module makes sense because Meyer shows that the ring $C_c^\infty(G)$ satisfies this property over itself, i.e.:
\[
    C_c^\infty(G)\widehat{\otimes}_{C_c^\infty(G)}C_c^\infty(G)\xrightarrow{\sim}C_c^\infty(G).
\]

We will prove a strengthening of this statement (although, for simplicity, we consider less general groups $G$ compared to \cite{borno_quasi_unital_algs2}). Specifically, we will show that the use of the completed tensor product is unnecessary. More explicitly, all essential $C_c^\infty(G)$-modules $M$ also satisfy that the morphism
\[
    C_c^\infty(G)\otimes_{C_c^\infty(G)}M\ra M
\]
is already an isomorphism  of bornological vector spaces -- that is, no completion is necessary. We will refer to this property as \emph{smoothness} of a $C_c^\infty(G)$-module. Once again, this property will make sense because we will show that $C_c^\infty(G)$ is in fact \emph{quasi-unital}, i.e.:
\[
    C_c^\infty(G)\otimes_{C_c^\infty(G)}C_c^\infty(G)\xrightarrow{\sim}C_c^\infty(G).
\]
For more on our definitions of quasi-unitality and smoothness, we refer the reader to Section~\ref{sect:quasi_unital}. Note that our terminology differs somewhat from that of \cite{self_induced_algs}.

In short, we claim that essential $C_c^\infty(G)$-modules are automatically smooth $C_c^\infty(G)$-modules, and establish an equivalence:
\begin{equation} \label{eq:smooth_is_smooth} \xymatrix{
    \{\text{smooth $G$-modules}\} \ar@{<->}[r] & \{\text{smooth $C_c^\infty(G)$-modules}\}.
}\end{equation}

In order to prove this remarkable property, we will use a strengthened version of the theorem of Dixmier-Malliavin. Specifically, we will show that the G\r{a}rding functor defined in Section~\ref{sect:garding_functor} defines an essential $C_c^\infty(G)$-module which is also a smooth $C_c^\infty(G)$-module. This will also allow us to prove some useful corollaries.
\begin{remark}
    Our notion of smoothness is such that a smooth bornological $C_c^\infty(G)$-module always satisfies that
    \[
        C_c^\infty(G)\otimes_{C_c^\infty(G)}M\ra M
    \]
    is an isomorphism of vector spaces (without any topology or bornology). Thus, our version of the Dixmier-Malliavin theorem implies the classical one.
\end{remark}

The structure of this text is as follows. In Section~\ref{sect:bornological_vect}, we will recall the notion of a bornological vector space and its basic properties. Additionally, we will linger on the comparison between bornological and topological vector spaces. In Section~\ref{sect:quasi_unital}, we will recall the notions of a quasi-unital algebra and its smooth modules. In Section~\ref{sect:garding_functor}, we will define the G\r{a}rding functor on bornological representations, which will allow us to define the notion of a smooth $G$-module. Armed with a definition for the two categories involved in the equivalence \eqref{eq:smooth_is_smooth}, in Section~\ref{sect:dixmier_malliavin} we will finally prove the desired equivalence, and establish the Dixmier-Malliavin theorem in the bornological setting. In Appendix~\ref{appendix:ext_of_quasi_unital}, we will prove a result from Section~\ref{sect:quasi_unital} that is not needed for the proof of our main theorem.

\date{\textbf{Acknowledgements: }The author would like to thank his advisor, Joseph Bernstein, for pointing out the Dixmier-Malliavin theorem to him, and his great help in refining this paper. The author would also like to thank Shachar Carmeli for many fruitful discussions about the category of bornological vector spaces.}

\section{Bornological Vector Spaces} \label{sect:bornological_vect}

In this section, we will introduce some background knowledge on bornological vector spaces for the reader who is not necessarily familiar with their use, and try to motivate their relevance to our work. This section should not be considered original contribution.

This section is far too short to be a comprehensive introduction to the theory of bornological vector spaces. However, the author hopes to emphasize both the added convenience that they give over using topological vector spaces, and their intuitive sense.

We will begin with a general background on bornological vector spaces and their comparison with topological vector spaces. Afterwards, we will discuss bornological \emph{modules}, and show some of their advantages over topological modules.

\begin{warning}
    The reader should beware that in this text, we are not automatically assuming that our bornological vector spaces are complete. Likewise, we will take care to distinguish between the (incomplete) tensor product $\otimes$, and the completed tensor product $\widehat{\otimes}$.
\end{warning}

\subsection{Background}

Generally speaking, just like a topological space specifies which of its subsets is \emph{open}, a bornological space is supposed to specify which of its subsets is \emph{bounded}. While in most circumstances topological and bornological spaces are fairly different notions, their specializations to the case of vector spaces are more closely related than might be expected.

Let us begin by introducing some notation.
\begin{definition}
    We denote:
    \begin{enumerate}
        \item By $\Vect$ the category of vector spaces over $\CC$.
        \item By $\SNorm$ the category of semi-normed topological vector spaces over $\CC$.
        \item By $\Top$ the category of locally convex topological vector spaces over $\CC$.
        \item By $\Born$ the category of bornological vector spaces, with convex vector bornologies over $\CC$.
    \end{enumerate}
    From now on, we will simply use the terms \emph{topological vector space} and \emph{bornological vector space} to refer to objects of $\Top$ and $\Born$, respectively.
\end{definition}

\begin{remark}
    Maps $f\co W\ra W'$ in $\Top$ are always \emph{continuous}: the inverse image $f^{-1}(U)$ of an open subset $U\subset W'$ is open in $W$. Likewise, maps $f\co V\ra V'$ in $\Born$ are always \emph{bounded}: the image $f(T)$ of a bounded subset $T\subseteq V$ is bounded in $V'$.
\end{remark}

\begin{remark}
    Just like a topological vector space can be described by a generating family of absolutely convex open subsets, a bornological vector space can be described by a generating family of bounded subsets.
    
    More explicitly, let $V$ be a vector space, and $T$ an absolutely convex subset. Then $T$ defines a semi-norm $\norm{\cdot}_T$ on its span $V_T$.
    
    Given a vector space $V$, one can specify a bornology on it as follows. Consider a collection $\mathcal{T}$ of absolutely convex subsets of $V$, satisfying that:
    \begin{enumerate}
        \item The vector spaces $V_T$ with $T\in\mathcal{T}$ exhaust $V$.
        \item For every finite $\{T_1,\dots,T_n\}\subset\mathcal{T}$, there is some $T\in\mathcal{T}$ such that $T_i$ is bounded in $V_T$ for all $i$.
    \end{enumerate}
    The \emph{bornology on $V$ generated by $\mathcal{T}$} is specified by saying that a subset $T'\subseteq V$ is bounded if and only if it is bounded in some $V_T$, for $T\in\mathcal{T}$ in this collection. We say that $\mathcal{T}$ is a \emph{generating collection of bounded subsets of $V$} for this bornology. Every convex vector bornology can be obtained in this way.
\end{remark}

Recall that $V\in\Born$ is said to be \emph{complete} if each $V_T$ is complete, as $T$ goes over some generating collection of bounded subsets. If $V$ is complete, then the space $V_T$ is automatically complete for all absolutely convex bounded subsets $T$. A space $V\in\Born$ is called \emph{separated} if all $V_T$ are Hausdorff. A bornological vector space $V$ has a canonical completion $\widehat{V}$, given by the colimit
\[
    \widehat{V}=\colim_T\widehat{V_T}
\]
of the completions of the semi-normed spaces $V_T$.

There are standard fully faithful embeddings
\begin{align*}
    \Top & \subseteq\Pro(\SNorm) \\
    \Born & \subseteq\Ind(\SNorm),
\end{align*}
induced via adjunction from the inclusions $\SNorm\subseteq\Top$ and $\SNorm\subseteq\Born$.

\begin{remark}
    To be as explicit as possible, let us describe these embeddings and their essential images.
    
    The embedding $\Top\subseteq\Pro(\SNorm)$ is given by sending a topological vector spaces to the pro-system given by a generating system of semi-norms. The essential image of $\Top$ consists precisely of those pro-systems of semi-normed vector spaces whose underlying pro-systems of vector spaces are constant.
    
    Similarly, the embedding $\Born\subseteq\Ind(\SNorm)$ is given by sending $V\in\Born$ to the ind-system $\{V_T\}$, as $T$ goes over a generating collection of bounded subsets for $V$. The essential image of $\Born$ consists precisely of those ind-systems of semi-normed vector spaces whose underlying ind-systems of vector spaces can be chosen to have injective transition maps.
\end{remark}

Now, the canonical adjunction between $\Ind(\SNorm)$ and $\Pro(\SNorm)$ restricts to an adjoint pair
\[\xymatrix{
    F\co\Born \ar@<3pt>[r] & \Top\co G. \ar@<3pt>[l]
}\]

\begin{remark}
    Consider a topological vector space $W\in\Top$. The corresponding bornological space $G(W)$ has the same underlying vector space, equipped with the \emph{von-Neumann bornology}: a set $T\subseteq G(W)$ is bounded if it is bounded in all of the semi-norms of $W$.
    
    Similarly, for a bornological vector space $V\in\Born$, the corresponding topological vector space $F(V)$ has the same underlying vector space, equipped with the \emph{bornivorous topology}: a subset $U\subseteq F(V)$ is open if it contains a scalar multiple of every bounded set of $V$.
\end{remark}

\begin{remark}
    A subspace $V'\subseteq V$ of a bornological vector space $V$ acquires a subspace bornology. If $V$ is complete, then $V'$ is also complete if and only if it is closed in the bornivorous topology on $V$.
\end{remark}

\begin{remark}
    Let $W\in\Top$, and suppose that the bornivorous topology of the von-Neumann bornology of $W$ is once again the original topology. In other words, suppose that the co-unit
    \[
        FG(W)\ra W
    \]
    is an isomorphism. In such a case, $W$ is sometimes referred to as a \emph{bornological topological vector space}. We will avoid using this term, as it might be very confusing.
\end{remark}

Regardless, there is a large fully faithful subcategory on both sides of the adjunction
\[
    \xymatrix{\Born \ar@<3pt>[r] & \Top \ar@<3pt>[l]}
\]
on which the adjunction defines an equivalence. Specifically, all \emph{metrizable} topological vector spaces (those whose topology is generated by countably many semi-norms) satisfy this property. I.e., metrizable topological vector spaces can be thought of as bornological vector spaces with no risk.

Because many interesting topological vector spaces encountered in practice are metrizable, it is easy to mistakenly try to generalize results that are actually about bornological vector spaces to results about topological vector spaces. This is exacerbated by the fact that while topological vector spaces are much more commonly used, bornological vector spaces are much better behaved. The representation theory of locally compact groups is one area where this phenomenon is noticeable. We simply reference the reader to \cite{borno_quasi_unital_algs2}, and possibly \cite{borno_quasi_unital_algs}, to illustrate this point.

\subsection{Tensor Products and Modules}

Having discussed the basic properties of bornological vector spaces, we wish to discuss the properties of bornological \emph{modules}, which are one side of the categorical equivalence~\eqref{eq:smooth_is_smooth}, the main goal of this paper.

In order to have a notion of module in $\Born$, we first need a notion of tensor product of bornological vector spaces. In this text, we will exclusively use the \emph{projective} tensor product for bornological vector spaces. Recall that for $V,V'\in\Born$, the projective tensor product $V\otimes V'$ is a bornological vector space, whose underlying vector space is the tensor product $V\otimes_\Vect V'$ of the underlying vector spaces of $V$ and $V'$, and whose bornology is as follows. The bornology of $V\otimes V'$ is generated by the absolute convex hulls of the subsets of the form $T\otimes T'$, where $T,T'$ go over a generating collection of bounded subsets for $V,V'$, respectively.

\begin{remark}
    If $V,V'\in\Born$ are complete, then it also makes sense to consider the bornological space $V\widehat{\otimes}V'$, which is defined as the completion of $V\otimes V'$. Note that the underlying vector space of $V\widehat{\otimes}V'$ is no longer the same as the tensor product of the underlying vector spaces of $V$ and $V'$.
\end{remark}

The tensor product $\otimes$ turns $\Born$ into a closed, unital symmetric monoidal category. In particular, the functor $\otimes$ respects colimits in both arguments, and $\Born$ has an internal $\Hom$ functor. The same is simply not true for $\Top$.
\begin{remark}
    The internal $\Hom$ functor on bornological vector spaces can also be described explicitly. If $V,V'\in\Born$, then the space $\Hom(V,V')$ of bounded maps can be assigned a bornology as follows. A subset $S\subseteq\Hom(V,V')$ is bounded if and only if for all bounded subsets $T\subseteq V$, there exists a bounded subset $T'\subseteq V'$, such that all maps in $S$ map $T$ into $T'$. This is called the \emph{equibounded} bornology. It is easy to verify that the equibounded bornology turns $\Hom(V,-)$ into a right adjoint $\sHom(V,-)$ to $-\otimes V$.
\end{remark}

In this text, rings are not necessarily commutative or unital. A particular consequence of the above is that categories of modules of rings in $\Born$ are well-behaved. That is, let $R$ be a (non-unital) ring in $\Born$. There is a category $\Mod(R)$, and a forgetful functor
\[
    \Mod(R)\ra\Born.
\]
This functor automatically respects limits, being a right adjoint. However, because $\Born$ is closed monoidal, it also respects colimits. In particular, one can compute direct limits of bornological modules at the level of underlying bornological vector spaces.

\begin{remark}
    Equip $\Top$ with its own projective tensor product. The functor
    \[
        \xymatrix{\Born \ar[r] & \Top}
    \]
    from the previous subsection, as well as the forgetful functor $\Born\ra\Vect$, both respect the symmetric monoidal structure of $\Born$. In particular, a bornological ring $R$ is also a topological ring via the bornivorous topology, and also simply a ring over $\CC$. Likewise, a bornological module over $R$ is also a topological module.
\end{remark}

On the other side of the categorical equivalence~\eqref{eq:smooth_is_smooth}, we will be dealing with representations of groups. Let $G$ be any (possibly discrete) group. We will say that a \emph{bornological group representation} $(\pi,V)$ of the group $G$ (also referred to as a \emph{$G$-module}) is a bornological vector space $V\in\Born$ along with a map $\pi\co G\ra\Aut(V)$. This means that $G$ acts on $V$ via bounded maps. Of course, when $G$ is actually a Lie group, it is senseless to consider representations without some kind of continuity or smoothness requirement on the dependence of $\pi$ on $G$. We will have an in-depth discussion of what it means for a representation of a Lie group to be smooth in Section~\ref{sect:garding_functor} below.

\section{Quasi-Unital Rings} \label{sect:quasi_unital}

The goal of this section is to collect some basic information about quasi-unital rings. In practice, we will only ever be using the notions below in the setting of bornological vector spaces. However, when trying to establish intuition, it is useful to consider them in the simpler setting of vector spaces (without any topology or bornology). As a result, we will develop the theory for the categories $\Vect$ (of vector spaces over $\CC$) and $\Born$ (of bornological vector spaces) simultaneously.

In any case, the two theories are very nearly identical, the biggest difference being the potential use of a non-trivial exact structure on the category of bornological vector spaces. The reader who is not concerned with such issues can skip Subsection~\ref{subsect:quasi_unital_born} below entirely.

For a more detailed overview of this notion, see Meyer's work in \cite{self_induced_algs}. The reader should beware that Meyer only considers \emph{complete} bornological vector spaces, and thus takes their completed tensor products. This causes some differences in terminology.

\subsection{Definition of Quasi-Unital Rings in \texorpdfstring{$\Vect$}{Vect}} \label{subsect:quasi_unital_vect}

Recall that in this text, rings are not necessarily unital or commutative. Our basic objects of study for this section are quasi-unital rings and their smooth modules. Let us begin by introducing this notion for vector spaces.

Recall that if $R$ is a (non-unital) ring, and $M,N$ are right and left $R$-modules respectively, then the \emph{relative tensor product} of $M$ and $N$ over $R$ is denoted by $M\otimes_R N$ and defined as the quotient
\[\xymatrix{
    M\otimes R\otimes N \ar@<2pt>[r] \ar@<-2pt>[r] & M\otimes N.
}\]
\begin{definition} \label{def:ring_is_quasi_unital}
    A (non-unital) ring $R$ in $\Vect$ is called \emph{quasi-unital} if the canonical morphism
    \begin{equation*}
        R\otimes_R R\ra R
    \end{equation*}
    is an isomorphism.
\end{definition}

\begin{definition} \label{def:smooth_modules}
    Let $R$ be a quasi-unital ring in $\Vect$. A left $R$-module is called \emph{smooth}, if
    \[
        R\otimes_R M\ra M
    \]
    is an isomorphism. We use similar definitions for right $R$-modules and $R$-bi-modules.
\end{definition}

\begin{remark}
    Note that the terms ``approximate identity'' and ``quasi-unital algebra'' are used in the literature with similar, but not identical meanings to our own, cf. \cite{borno_quasi_unital_algs,borno_quasi_unital_algs2}. Also note that this meaning is very different from the use of the term ``quasi-unital algebra'' in the context of higher algebra, where it is used to refer to an algebra whose unit carries reduced coherence data. Our meaning is closer to the term ``self-induced algebra'' used by Meyer in \cite{self_induced_algs}, although Meyer uses the completed tensor product for bornological algebras, which we do not.
\end{remark}

Let us begin by giving some examples.

\begin{example} \label{example:unital_means_quasi_unital}
    All unital rings are quasi-unital. In fact, having a one-sided unit is enough for a ring to be quasi-unital. Indeed, it is easy to find a splitting for the complex
    \[\xymatrix{
        \dots \ar[r] & R\otimes R\otimes R \ar[r] & R\otimes R \ar[r] & R \ar[r] & 0
    }\]
    in this case.
\end{example}

\begin{example} \label{example:quasi_unital_ring}
    Let
    \[
        R=\CC\left[t^\alpha\suchthat 0<\alpha\in\RR\right]
    \]
    be the commutative ring of expressions of the form
    \[
        \sum_{i=0}^N a_i t^{\alpha_i}
    \]
    for real $\alpha_i>0$. It is immediate to see that $R$ is quasi-unital.
\end{example}

\begin{example} \label{example:smooth_module}
    Continuing Example~\ref{example:quasi_unital_ring}, we see that the $R$-module $M_{(0,1]}$ given by the quotient
    \[
        M_{(0,1]}=R/tR
    \]
    is smooth. This module is spanned by powers $t^\alpha$, where $\alpha\in(0,1]$.
\end{example}

\begin{example} \label{example:non_smooth_module}
    Let us give two examples for non-smooth modules. Let $R$ be as in Example~\ref{example:quasi_unital_ring}. We define two $R$-modules $M_{[0,\infty)}$ and $M_{(0,1)}$ by:
    \begin{align*}
        M_{[0,\infty)} & =\CC\cdot 1\oplus R, \\
        M_{(0,1)} & =R/(\CC t+Rt)=R/tM_{[0,\infty)}.
    \end{align*}
    That is, $M_{[0,\infty)}$ is given by adjoining a unit to $R$, and $M_{(0,1)}$ is the maximal quotient of $R$ where $t=0$. They are spanned by powers $t^\alpha$, where $\alpha\in[0,\infty)$ or $\alpha\in(0,1)$, respectively.
    
    Now, the map
    \[
        R\otimes_R M_{[0,\infty)}=R\ra M_{[0,\infty)}
    \]
    is injective but not surjective, meaning that $M_{[0,\infty)}$ is not smooth. Similarly,
    \[
        R\otimes_R M_{(0,1)}\ra M_{(0,1)}
    \]
    is surjective but not injective. For example, the element $t^{1/2}\otimes t^{1/2}$ lies in its kernel. In particular, $M_{(0,1)}$ is not smooth either.
\end{example}

\subsection{Definition of Quasi-Unital Rings in \texorpdfstring{$\Born$}{Born}} \label{subsect:quasi_unital_born}

Let us now address the issue of bornologies. The main reason why we cannot simply re-use the definition above is as follows. In many contexts, the notion of quotient in the category $\Born$ is a bit too weak. This means that although
\[\xymatrix{
    R\otimes R\otimes R \ar@<2pt>[r] \ar@<-2pt>[r] & R\otimes R \ar[r] & R
}\]
is a colimit diagram, one might want to further require that $R\otimes R\ra R$ be an admissible epimorphism, for some exact structure on $\Born$.

The reader who does not care about such issues may take Definitions~\ref{def:ring_is_quasi_unital} and~\ref{def:smooth_modules} to apply verbatim to bornological rings as well, and skip the rest of this subsection. All theorems and statements below will still hold. However, this would result in a slightly incorrect notion. The rest of this subsection contains a detailed discussion of this issue.

To begin with, the reader should be aware that choosing the right notion of admissible epimorphism for bornological vector spaces seems to be subtle, with conflicting choices made in the literature. Compare, for example, \cite{derived_borno_spaces} (which uses \emph{refined cohomology} to test for exactness) and \cite{borno_quasi_unital_algs} (which requires that an exact sequence of bornological vector spaces be \emph{split}).

In this text, we will make the minimally powerful choice under which our main theorem is still true. To be specific, we will exclusively use the locally split exact structure for bornological vector spaces for the rest of this text:
\begin{definition}
    Let $\pi\co V\ra V'$ be a map of bornological vector spaces. We will say that $\pi$ is \emph{locally split} if for all bounded subsets $T\subseteq V'$, there exists a bounded linear section $s_T\co V'_T\ra V$:
    \[\xymatrix{
        V \ar[r]^\pi & V' \\
        & V'_T. \ar@{^{(}->}[u] \ar@{-->}[ul]^{s_T}
    }\]
\end{definition}
\begin{remark}
    Note that we require no compatibility between sections $s_T$ corresponding to different bounded subsets $T\subseteq V'$. In particular, there might not exist a section to $\pi$ itself.
\end{remark}

\begin{definition} \label{def:locally_split_exact}
    More generally, we will say that a complex
    \[\xymatrix{
        \dots \ar[r]^{d^{-1}} & V^0 \ar[r]^{d^0} & V^1 \ar[r]^{d^1} & \dots
    }\]
    of bornological vector spaces is \emph{exact} at $V^i$ if for every bounded subset $T\subseteq\ker(d^i)$, there exists a section $s^i_T\co\ker(d^i)_T\ra V^{i-1}$ on the span of $T$ making the following diagram commute:
    \[\xymatrix{
        V^{i-1} \ar[r]^{d^{i-1}} & V^i \\
        & \ker(d^i)_T. \ar@{^{(}->}[u] \ar@{-->}[ul]^{s^i_T}
    }\]
\end{definition}

\begin{warning}
    Note that if $\xymatrix@1{0 \ar[r] & V' \ar[r] & V \ar[r] & V'' \ar[r] & 0}$ is a short exact sequence of bornological vector spaces, then their underlying vector spaces also form a short exact sequence. This might subvert the expectations of some readers who are used to working with complete topological vector spaces, where the implication is typically the other way around.
\end{warning}

We can now state our new definitions:
\begin{definition}
    A bornological ring $R$ is called \emph{quasi-unital} if the canonical complex
    \[\xymatrix{
        \dots \ar[r] & R\otimes R\otimes R \ar[r] & R\otimes R \ar[r] & R \ar[r] & 0
    }\]
    is exact at $R$ and $R\otimes R$, in the sense of Definition~\ref{def:locally_split_exact}.
\end{definition}

\begin{definition} \label{def:born_smooth_modules}
    Let $R$ be a quasi-unital bornological ring. A bornological left $R$-module is called \emph{smooth}, if the canonical complex
    \[\xymatrix{
        \dots \ar[r] & R\otimes R\otimes M \ar[r] & R\otimes M \ar[r] & M \ar[r] & 0
    }\]
    is exact at $M$ and $R\otimes M$. We use similar definitions for right $R$-modules and $R$-bi-modules.
\end{definition}

\begin{remark}
    The functors $\Born\ra\Top$ and $\Born\ra\Vect$ respect colimits. Hence, rings $R$ that are quasi-unital as bornological rings are also quasi-unital as rings in $\Top$ or $\Vect$, and the same applies to smooth modules.
    
    This allows one to deduce the classical Dixmier-Malliavin theorem from our bornological version, Theorem~\ref{thm:garding_is_smooth}.
\end{remark}

\begin{remark}
    Using incomplete tensor products in Definition~\ref{def:born_smooth_modules} makes a big difference; if $M$ is a smooth bornological $R$-module, then every element $m\in M$ can be written down as a finite sum $\sum_i r_i\cdot m_i$, for $r_i\in R$ and $m_i\in M$. Had we been using the completed tensor product $\widehat{\otimes}$ instead, we would not necessarily have had this property.
\end{remark}

\subsection{Properties of Quasi-Unital Rings}

We now return to our discussion of quasi-unital rings, and their properties. For the rest of this subsection, all rings and modules are bornological, although we make no use of this fact.

Let us recall our main motivation.
\begin{example} \label{example:smooth_compact_funcs_are_quasi_unital}
    The main point of this text is that many rings and modules naturally appearing in representation theory are quasi-unital and smooth, respectively. For example, for a real Lie group $G$, the ring $C_c^\infty(G)$ of smooth and compactly supported functions on $G$, equipped with the convolution product, is a quasi-unital ring. However, it is generally not unital unless $G$ is discrete.
    
    Similarly, if $G$ acts smoothly on a space $X$, then the space $C_c^\infty(X)$ becomes a smooth module over $C_c^\infty(G)$. Both of these statements will follow from Theorem~\ref{thm:garding_is_smooth} below.
\end{example}

\begin{definition} \label{def:smoothening_roughening}
    Let $R$ be a quasi-unital bornological ring. We use the notation $\Mod^\text{sm}(R)$ to denote the full subcategory of smooth left $R$-modules inside the category $\Mod(R)$ of bornological left $R$-modules. The inclusion $\Mod^\text{sm}(R)\subseteq\Mod(R)$ has a right adjoint: 
    \begin{align*}
        \Mod(R) & \ra \Mod^\text{sm}(R) \\
        M & \mapsto R\otimes_R M
    \end{align*}
    called the \emph{smoothening} of $M$.
\end{definition}

\begin{remark}
    The smoothening functor $M\mapsto R\otimes_R M$ itself has a right adjoint, given by the internal $\Hom$ space $M\mapsto\sHom_R(R,M)$. We will refer to this as the \emph{roughening} of $M$.
\end{remark}

\begin{remark} \label{remark:bmod_is_unital}
    Note that the category $\BMod^\text{sm}(R)$ of bornological $R$-bi-modules that are smooth on both sides is a unital symmetric monoidal category, with tensor product given by
    \[
        (M,N)\mapsto M\otimes_R N
    \]
    and unit given by the object $R$.
\end{remark}

For unital ring extensions $R\ra R'$, the unit of $R$ induces the unit of $R'$. It turns out that despite not actually having a unit, quasi-unital rings have a similar property.
\begin{claim} \label{claim:ext_of_quasi_unital}
    Let $R$ be a bornological quasi-unital ring, and let $R\ra R'$ be an extension of $R$. Suppose that $R'$ is smooth as an $R$-bi-module. Then $R'$ is a quasi-unital ring.
    
    Moreover, if in the situation above $M$ is a bornological $R'$-module which is smooth as an $R$-module, then $M$ is also smooth as an $R'$-module.
\end{claim}
\begin{remark} \label{remark:quasi_unital_is_transitive}
    Claim~\ref{claim:ext_of_quasi_unital} allows us to extend Example~\ref{example:smooth_compact_funcs_are_quasi_unital} to many more naturally occurring objects in representation theory. For example, the space of Schwartz functions $S(G)$ on a real Lie group $G$ is a smooth $G$-bi-module, thus a smooth $C_c^\infty(G)$-bi-module by the main result of this text (Theorem~\ref{thm:garding_is_smooth}), and hence a quasi-unital ring itself by Claim~\ref{claim:ext_of_quasi_unital}.
\end{remark}

Because its proof is technically involved, and not particularly informative, we will postpone the proof of Claim~\ref{claim:ext_of_quasi_unital} to Appendix~\ref{appendix:ext_of_quasi_unital}.

\section{The G\r{a}rding Functor} \label{sect:garding_functor}

Let $G$ be a real Lie group. Our immediate goal is to define a notion of \emph{smoothness} for bornological $G$-modules. We will later use this notion, both to define various spaces of smooth functions with compact support, as well as state our main theorem. The notion we describe here is equivalent to that of \cite{borno_quasi_unital_algs2}, and should not be considered original.

Our approach to smooth functions is via defining a \emph{G\r{a}rding functor}, which takes a space of functions into the space of smooth functions inside it. Thus, for example, the space $C_c^\infty(G)$ of smooth functions with compact support will be the result of applying the G\r{a}rding functor to the space of continuous functions with compact support (see also Construction~\ref{const:C_c_infty}).

Let us begin by defining the G\r{a}rding functor, and proving its basic properties. We will denote by $\mathfrak{g}$ the Lie algebra of $G$, and by $U(\mathfrak{g})$ its universal enveloping algebra.

\begin{definition}
    A \emph{smoothness bound} for $G$ is a norm on the vector space $U(\mathfrak{g})$:
    \[
        \lambda\co U(\mathfrak{g})\ra\RR_{>0}.
    \]
\end{definition}
Note that we do not ask $\lambda$ to have any special properties with respect to the multiplicative structure on $U(\mathfrak{g})$. Intuitively, a smoothness bound tells us ``how smooth'' a specific function is -- how large are each of its derivatives.

The collection of smoothness bounds for $G$ has a partial order, going from from smoothest to least smooth:
\begin{definition}
    We say that a smoothness bound $\lambda_1$ is \emph{smoother than} the smoothness bound $\lambda_2$, and write $\lambda_1\leq\lambda_2$, if and only if $\norm{D}_{\lambda_1}\leq \norm{D}_{\lambda_2}$ for all $D\in U(\mathfrak{g})$.
\end{definition}
Note that this partial order is filtered.

We can now define the G\r{a}rding functor, intuitively given by the space of all vectors whose derivatives are bounded by some smoothness bound:
\begin{construction}
    Let $V$ be a complete bornological $G$-module, and let $\lambda$ be a smoothness bound. We will construct (functorially in $V$ and $\lambda$) a complete bornological space $V^{(\lambda)}$, along with an injective map of bornological spaces $V^{(\lambda)}\ra V$. The \emph{G\r{a}rding space} of $V$ will be the bornological $G$-module
    \[
        V^{(\infty)}=\colim_{\lambda} V^{(\lambda)},
    \]
    along with the natural map $V^{(\infty)}\ra V$. Here, the colimit goes over the filtered poset of all smoothness bounds for $G$.
    
    To construct $V^{(\lambda)}$, we will do the following. For each absolutely convex bounded subset $T\subseteq V$, we will construct a bornological space $V_T^{(\lambda)}$, along with an injective map $V_T^{(\lambda)}\ra V_T$. We will let
    \[
        V^{(\lambda)}=\colim_T V_T^{(\lambda)}.
    \]
    
    So, let $T\subseteq V$ be an absolutely convex bounded subset. We define $V_T^{(\lambda)}$ to be the span of the convex subset of all vectors $v\in V_T$ such that the derivatives $D\cdot v$ are well-defined with respect to the semi-norm of $V_T$ for all $D\in U(\mathfrak{g})$, and satisfy
    \[
        \norm{D\cdot v}_T\leq\norm{D}_\lambda.
    \]
    Note that, in particular, $\norm{v}_T\leq\norm{1}_\lambda<\infty$.
    
    Finally, we give $V_T^{(\lambda)}$ the semi-norm corresponding to the above convex subset.
\end{construction}

\begin{remark}
    One can build a similar construction for an algebraic group over a local field $G=G(F_v)$, or over an adelic group $G(\AA)$. This can be done by adding to the smoothness bound $\lambda$ the data of a compact open subset $K\subseteq G$, and requiring that the vectors of $V^{(\lambda,K)}$ also be invariant under $K$. We leave the details to the reader.
\end{remark}

\begin{claim} \label{claim:garding_nice_to_smooth}
    Let $V$ be a complete bornological $G$-module, and denote $W=V^{(\infty)}$. Then the natural map $W^{(\infty)}\ra W$ is an isomorphism.
\end{claim}
\begin{proof}
    Let $\lambda$ be a smoothness bound. It is enough to show that there exists a smoothness bound $\lambda'$ such that $V^{(\lambda)}$ is contained in
    \[
        \left(V^{(\lambda')}\right)^{(\lambda')}.
    \]
    Indeed, we see that it is sufficient to ask for
    \[
        \norm{D\cdot D'}_\lambda\leq\norm{D}_{\lambda'}\norm{D'}_{\lambda'}
    \]
    for all $D,D'\in U(\mathfrak{g})$.
    
    The simplest way to show that such a $\lambda'$ necessarily exists is as follows. Pick a basis $\{g_i\}$ for $\mathfrak{g}$, and use it to write a monomial basis for $U(\mathfrak{g})$. Now, there is a sequence $\{A_d\}$ of positive real numbers such that:
    \[
        \norm{D}_{\lambda}\leq A_{d},
    \]
    for all $D$ in the monomial basis of degree $\leq d$. We may suppose without loss of generality that the sequence $\{A_d\}$ is monotone increasing, and $A_0\geq 1$.
    
    Furthermore, there is some sequence $\{B_d\}$ of positive real numbers, depending only on the size of the coefficients of the decomposition of elements of the form $[g_i,g_j]$ into the basis $\{g_k\}$, such that the following holds. For any two elements $D,D'$ of the monomial basis, each of degree $\leq d$,
    \[
        \norm{D\cdot D'}_{\lambda}\leq B_d A_{2d}.
    \]
    Without loss of generality, we also assume that the sequence $\{B_d\}$ is monotone increasing, and $B_0\geq 1$.
    
    We set $C_d=B_d A_{2d}$, and define:
    \[
        \norm{\sum\alpha_i D_i}_{\lambda'}=\sum C_{\deg{D_i}}\abs{\alpha_i},
    \]
    where the $D_i$ are elements of the monomial basis, with degrees $\deg{D_i}$. It is clear that the smoothness bound $\lambda'$ does the trick.
\end{proof}

Let $\Mod^\comp(G)$ be the category of complete bornological $G$-modules. By Claim~\ref{claim:garding_nice_to_smooth}, we conclude that the G\r{a}rding functor defines a right adjoint to the embedding of its essential image in $\Mod^\comp(G)$. Thus, we define:
\begin{definition} \label{def:smooth_G_module}
    We let $\Mod^{\comp,\sm}(G)\subseteq\Mod^\comp(G)$ be the full subcategory given by the essential image of the G\r{a}rding functor. We will refer to an object $V\in\Mod^{\comp,\sm}(G)$ as a \emph{smooth $G$-module}.
\end{definition}

\begin{remark}
    It is easy to see that this notion of smoothness of $G$-modules coincides with that of \cite{borno_quasi_unital_algs2}, and the G\r{a}rding functor coincides with the smoothening functor of $G$-modules of \cite{borno_quasi_unital_algs2}.
\end{remark}

This lets us define some spaces of smooth functions:
\begin{construction} \label{const:C_c_infty}
    We denote by $C_c^\infty(G)$ the result of applying the G\r{a}rding functor to the complete bornological left $G$-module $L^\infty_c(G)$. Here, $L^\infty_c(G)$ is the colimit
    \[
        L^\infty_c(G)=\colim_{K\subseteq G}L_K^\infty(G)
    \]
    of the $L^\infty$ function spaces supported in $K$, as $K$ goes over all compact subsets of $G$.
\end{construction}

\begin{remark}
    The space $C_c^\infty(G)$ admits $G$-actions from both left and right. By definition, it is smooth from the left. However, it is also smooth \emph{from the right}.
\end{remark}

\begin{remark}
    It is well known that $C_c^\infty(G)$ is a (non-unital) ring, with multiplication given by the convolution product. In Section~\ref{sect:dixmier_malliavin}, we will show that $C_c^\infty(G)$ is in fact a \emph{quasi-unital} ring, i.e. that
    \[
        C_c^\infty(G)\otimes_{C_c^\infty(G)}C_c^\infty(G)\xrightarrow{\sim}C_c^\infty(G),
    \]
    as a remnant of the G\r{a}rding construction.
\end{remark}

\section{Proof of Dixmier-Malliavin} \label{sect:dixmier_malliavin}

Recall the notions of a quasi-unital ring and a smooth module for such a ring from Section~\ref{sect:quasi_unital}. Our main theorem for this section is:
\begin{theorem} \label{thm:garding_is_smooth}
    The ring $C_c^\infty(G)$ is quasi-unital.
    
    Moreover, let $V$ be a complete bornological $G$-module. Then $V^{(\infty)}$ is a smooth $C_c^\infty(G)$-module.
\end{theorem}

\begin{remark}
    In Theorem~\ref{thm:garding_is_smooth} above, the bornological vector space $V^{(\infty)}$ is a-priori a $G$-module rather than a $C_c^\infty(G)$-module. However, it is a standard result that the action of $G$ on a complete smooth $G$-module can be integrated to an action of $C_c^\infty(G)$. See, for example, the equivalence of \cite{borno_quasi_unital_algs2}.
\end{remark}

\begin{remark} \label{remark:adelic_garding_is_smooth}
    A variant of Theorem~\ref{thm:garding_is_smooth} applies to adelic groups $G=G(\AA)$, with the G\r{a}rding functor appropriately modified. A similar statement is true for groups over non-Archimedean local fields.
\end{remark}

We can also reformulate Theorem~\ref{thm:garding_is_smooth} as:
\begin{corollary}
    Let $V$ be a complete smooth $G$-module. Then $V$ is a smooth $C_c^\infty(G)$-module.
\end{corollary}

Let us postpone the proof of Theorem~\ref{thm:garding_is_smooth} to the end of this subsection, in order to show some interesting corollaries:

\begin{corollary} \label{cor:comp_of_smooth_is_smooth}
    Let $M$ be a smooth $C_c^\infty(G)$-module. Then the completion $\widehat{M}$ is also smooth.
\end{corollary}
\begin{proof}
    We note that $\widehat{M}^{(\infty)}\ra\widehat{M}$ is an equivalence. Thus Theorem~\ref{thm:garding_is_smooth} implies the desired result.
\end{proof}

\begin{corollary} \label{cor:subspace_of_smooth_smooth}
    Let $M$ be a complete smooth $C_c^\infty(G)$-module, and let $N\subseteq M$ be a closed $C_c^\infty(G)$-submodule, with the induced bornology. Then $N$ is also smooth.
\end{corollary}
\begin{proof}
    It suffices to show that $N^{(\infty)}\ra N$ is an equivalence. However, we see that the diagram
    \[\xymatrix{
        N^{(\infty)} \ar[r] \ar[d] & N \ar[d] \\
        M^{(\infty)} \ar[r] & M \\
    }\]
    is Cartesian. By Claim~\ref{claim:garding_nice_to_smooth}, we are done.
\end{proof}

\begin{remark}
    As shown by the embedding $tM_{[0,\infty)}\subseteq R$ whose cokernel was used in Example~\ref{example:non_smooth_module}, Corollary~\ref{cor:subspace_of_smooth_smooth} does not hold for general quasi-unital rings.
\end{remark}

As another corollary, we obtain that the subtlety we introduced into the definition of smooth modules in Subsection~\ref{subsect:quasi_unital_born} is not actually necessary for $C_c^\infty(G)$-modules. That is, exactness in the sense of Definition~\ref{def:locally_split_exact} is automatic.
\begin{corollary} \label{cor:loc_split_exact_auto}
    Let $M$ be a complete $C_c^\infty(G)$-module. Then the following are equivalent:
    \begin{enumerate}
        \item \label{item:M_is_smooth} The $C_c^\infty(G)$-module $M$ is smooth.
        \item \label{item:M_is_weak_smooth} The canonical morphism
        \[
            C_c^\infty(G)\otimes_{C_c^\infty(G)}M\ra M
        \]
        is an isomorphism of bornological vector spaces.
        \item \label{item:M_is_essential} The canonical morphism
        \[
            C_c^\infty(G)\widehat{\otimes}_{C_c^\infty(G)}M\ra M
        \]
        is an isomorphism of bornological vector spaces.
    \end{enumerate}
\end{corollary}
\begin{proof}
    The implications $\ref{item:M_is_smooth}\implies\ref{item:M_is_weak_smooth}\implies\ref{item:M_is_essential}$ are immediate. Let us prove that $\ref{item:M_is_essential}\implies\ref{item:M_is_smooth}$. We are given the isomorphism
    \[
        C_c^\infty(G)\widehat{\otimes}_{C_c^\infty(G)}M\xrightarrow{\sim} M.
    \]
    This implies that $M$ is a smooth $G$-module, i.e. that $M^{(\infty)}\xrightarrow{\sim}M$. In particular, Theorem~\ref{thm:garding_is_smooth} now implies that $M$ is also a smooth $C_c^\infty(G)$-module.
\end{proof}

\begin{remark}
    Using Theorem~\ref{thm:garding_is_smooth} and Claim~\ref{claim:ext_of_quasi_unital}, it is possible to show that many standard spaces of functions (such as spaces of Schwartz functions) are also quasi-unital. See also Remark~\ref{remark:quasi_unital_is_transitive}.
\end{remark}

Finally, let us return to the proof of Theorem~\ref{thm:garding_is_smooth}.
\begin{proof}[Proof of Theorem~\ref{thm:garding_is_smooth}]
    We may suppose that $V^{(\infty)}\xrightarrow{\sim}V$. That is, we assume that $V$ is a smooth $G$-module.
    
    We will largely follow the ideas of \cite{schwartz_is_quasi_unital}. Reformulating the proof of the main theorem there, it supplies a lift
    \[\xymatrix{
        & V^{(\lambda)} \ar[d] \ar@{-->}[ld] \\
        C_c^\infty(G)\otimes V \ar[r] & V
    }\]
    for every smoothness bound $\lambda$ for $G$. Additionally, this lift is functorial in the smooth $G$-module $V$, albeit \emph{not} in the smoothness bound $\lambda$.
    
    Regardless, functoriality in $V$ means that this lift locally splits the augmented semi-simplicial object
    \[\xymatrix{
        \dots \ar@<3pt>[r] \ar[r] \ar@<-3pt>[r] & C_c^\infty(G)\otimes C_c^\infty(G)\otimes V \ar@<2pt>[r] \ar@<-2pt>[r] & C_c^\infty(G)\otimes V \ar[r] & V,
    }\]
    meaning that
    \[\xymatrix{
        C_c^\infty(G)\otimes_{C_c^\infty(G)}V \ar[r]^-\sim & V
    }\]
    is an isomorphism.
\end{proof}

\begin{appendices}
\section{Proof of Claim~\ref{claim:ext_of_quasi_unital}} \label{appendix:ext_of_quasi_unital}

The goal of this appendix is to prove Claim~\ref{claim:ext_of_quasi_unital}. Recall that we are trying to show that smooth extensions of quasi-unital rings are also quasi-unital.
\begin{proof}[Proof of Claim~\ref{claim:ext_of_quasi_unital}]
    The statements about algebras and their modules follow using almost identical arguments. Therefore, we will write down the proof that $R'$ is quasi-unital, and leave the smoothness of $M$ to the reader.
    
    Philosophically, the reason that $R'$ is quasi-unital is that $R'$ is a unital algebra object in the monoidal category $\BMod^\text{sm}(R)$ (see Remark~\ref{remark:bmod_is_unital}), and thus it is also quasi-unital by Example~\ref{example:unital_means_quasi_unital}. This argument can be formalized in a relatively clean way, but the use of non-trivial exact structures means that it requires the use of tools that are beyond the scope of this text.
    
    Therefore, we will directly exhibit this instead. We need to show that the complex
    \begin{equation} \label{eq:R_ext_complex} \xymatrix{
        R'\otimes R'\otimes R' \ar[rr]^-{m\otimes\id_{R'}-\id_{R'}\otimes m} & & R'\otimes R' \ar[r]^-{m} & R' \ar[r] & 0
    }\end{equation}
    is exact at $R'$ and $R'\otimes R'$. The first of these is easy: because $R'$ is smooth as a left $R$-module, there exist (locally on $R'$) sections into $R\otimes R'$. We compose these sections with the map $i\otimes\id_{R'}\co R\otimes R'\ra R'\otimes R'$ to show that the complex \eqref{eq:R_ext_complex} is exact at $R'$. Here, $i\co R\ra R'$ is the extension map.
    
    It remains to show exactness at $R'\otimes R'$. We want to define (locally) a section from the kernel of $m\co R'\otimes R'\ra R'$ to $R'\otimes R'\otimes R'$. In fact, we will show that it is possible to find a (local) section $\ker(m)\ra R'\otimes R'\otimes R'$ that factors through $i\otimes\id_{R'}\otimes\id_{R'}\co R\otimes R'\otimes R'\ra R'\otimes R'\otimes R'$. That is, we want a map (locally defined on the source)
    \[
        \eta\co\ker(m)\ra R\otimes R'\otimes R'
    \]
    whose composition with $(m\otimes\id_{R'}-\id_{R'}\otimes m)\circ(i\otimes\id_{R'}\otimes\id_{R'})$ is the identity.
    
    This is a more complicated diagram chase, although as usual for such chases given in printed form, the notation is probably more complicated than the chase itself. For the rest of this proof, we will use a dashed arrow ($\xymatrix@1{\ar@{-->}[r] &}$) to denote a map that is only defined locally. Then we have the following diagram:
    \begin{equation} \label{eq:mult_for_ring_ext} \xymatrix{
        R\otimes R'\otimes R' \ar[d]^{\id_R\otimes m} \ar[r]^-{a\otimes\id_{R'}} & R'\otimes R' \ar[d]^m \ar@{-->}@/_1.7pc/[l]_{s_0\otimes \id_{R'}}\\
        R\otimes R' \ar[r]^-a & R' \ar@{-->}@/^1.7pc/[l]_{s_0},
    }\end{equation}
    with $a\co R\otimes R'\ra R'$ being the action map. Locally on $R'$, the map $a$ has a section, which we will call $s_0$. The composition $(\id_R\otimes m)\circ (s\otimes\id_{R'})$ maps into $R\otimes R'$. If its image there were $0$, then $s_0\otimes\id_{R'}$ would work as the section $\eta$ we seek. However, we have no guarantee of that.
    
    Instead, we will need to ``correct'' $s_0\otimes\id_{R'}$ using the higher section from the smoothness of $R'$. Indeed, we note that by commutativity of the inner square of diagram~\eqref{eq:mult_for_ring_ext}, the counter-clockwise composition
    \[
        a\circ(\id_R\otimes m)\circ (s_0\otimes \id_{R'})
    \]
    is $m$. Thus, on the kernel of $m$, the map $(\id_R\otimes m)\circ (s_0\otimes \id_{R'})$ lies in the kernel of $a$, and therefore we can apply the section $s_1$ for 
    \begin{equation*} \xymatrix{
        R\otimes R\otimes R' \ar[r] & R\otimes R' \ar[r]^-{a} \ar@{-->}@/_1.7pc/[l]_{s_1} & R' \ar[r] \ar@{-->}@/_1.7pc/[l]_{s_0} & 0
    }\end{equation*}
    coming from the smoothness of $R'$. Our desired section is now
    \[
        \eta=\left.(s_0\otimes \id_{R'})\right|_{\ker(m)} - (\id_R\otimes i \otimes \id_{R'})\circ s_1\circ\left.(\id_R\otimes m)\circ (s_0\otimes \id_{R'})\right|_{\ker(m)},
    \]
    which maps the kernel of $m$ to $R\otimes R'\otimes R'$. Recall that $i\co R\ra R'$ is the extension map.
\end{proof}

\end{appendices}

\Urlmuskip=0mu plus 1mu\relax
\bibliographystyle{alphaurl}
\bibliography{L_func}

\end{document}